\documentclass[12pt]{amsart}
\usepackage{latexsym,enumerate}
\usepackage{amsmath,amsthm,amsopn,amstext,amscd,amsfonts,amssymb,mathrsfs}
\usepackage{color}

\usepackage{showidx}
\newcommand{\diint}{\displaystyle\int}
\newcommand{\dsum}{\displaystyle\sum}
\numberwithin{equation}{section}

\newtheorem{theorem}{Theorem}
\newtheorem{lemma}{Lemma}

\newtheorem{remark}{Remark}

\numberwithin{theorem}{section}
\numberwithin{corollary}{section}
\numberwithin{lemma}{section}
\numberwithin{definition}{section}
\numberwithin{proposition}{section}
\numberwithin{remark}{section}

\begin{document}

\title[]{Some weighted isoperimetric problems on $\mathbb{R}^N _+ $ with stable half balls have no solutions
}

\author[F. Brock]{F. Brock$^1$}
\author[F. Chiacchio]{F. Chiacchio$^2$}


\setcounter{footnote}{1}
\footnotetext{Swansea University, Department of 
Mathematics,  Computational Foundry, College of 
Science,
Swansea 
University Bay Campus, Fabian Way, Swansea SA1 8EN, Wales, UK, 
email:{ friedemann.brock@swansea.ac.uk}}

\setcounter{footnote}{2}
\footnotetext{Dipartimento di Matematica e Applicazioni ``R. Caccioppoli'',
Universit\`a degli Studi di Napoli Federico II,
Complesso Monte S. Angelo, via Cintia, 80126 Napoli, Italy;  
e-mail: {fchiacch@unina.it}}

\begin{abstract} 
We show the counter-intuitive fact that some weighted isoperimetric problems on the half-space $ \mathbb{R}^N _+  $, for which half-balls centered at the origin are stable, have  no solutions. 
A particular case is the measure $d\mu = x_N ^{\alpha } \, dx$, with $\alpha \in (-1,0)$. Some results on stability and nonexistence for weighted isoperimetric problems on $\mathbb{R}^N $ are also obtained.  

\medskip

\noindent
{\sl Key words: Isoperimetric inequality, Wirtinger inequality,  eigenvalue problem}  
\rm 
\\[0.1cm]
{\sl 2000 Mathematics Subject Classification:} 51M16, 46E35, 46E30, 35P15 
\rm 
\end{abstract}
\maketitle

\section{Introduction}
A \textsl{manifold with density} is a manifold endowed with a positive function, 
the \textsl{density}, which weights both the volume and the perimeter. 
This mathematical subject
is attracting an increasing attention from 
the mathematical community.
The related bibliography is very wide and, in this short note, it is impossible to give 
an exhaustive account of it. Hence we remind the interested reader to \cite{morganbook} and  \cite{MP}  and the references therein. 
One natural issue in this setting  consists of finding families of densities
for which one can determine the explicit form of the  isoperimetric set, 
see for instance \cite{S},  \cite{Bo}, \cite{MadernaSalsa}, \cite{CMV}, \cite{DHHT}, \cite{XR}, \cite{BCM},  \cite{BCM_Pot_Anal}, \cite{Cham}.

The problem becomes more challenging  when perimeter and volume carry two different weights. One important example is when the manifold is  $\mathbb{R}^N$, ($N\geq 2$), and the two weight functions are powers of the distance from the origin, see \cite{ABCMP2017}, and the references cited therein.
 The theorem proved  in \cite{ABCMP2017}  states that all spheres about the origin are isoperimetric for a certain range of the powers. One can modify this problem by inserting a further homogeneous perturbation term, namely $x_{N}^{\alpha}$, both in the volume and in the perimeter, see
\cite{ABCMP} and \cite{ABCMPapplana}:
\vspace{0.1cm} 
\begin{equation}
\text{Minimize }
 \int_{\partial \Omega }|x|^{k}x_{N}^{\alpha }\,{\mathcal{H}}_{N-1}(dx) \,\,
\text{among all smooth sets }
 \Omega \subset \mathbb{R}_{+}^{N}
\text{ satisfying }
{\int_{\Omega }|x|^{\ell}x_{N}^{\alpha}\,dx=1}  \tag{{\bf P}}
\end{equation}
\vspace{0.1cm}  
\noindent where   $\mathbb{R}^{N} _{+}  := \{ x \in \mathbb{R}^N :\, x_N >0\} $ and
  $k,\ell , \alpha  \in \mathbb{R}$.
 
Adapting some new methods introduced in \cite{ABCMP2017}, the authors find, for any  given positive number $\alpha $, a range of parameters $k$ and $\ell$
 for which the isoperimetric sets are   
intersections of balls centered  at the origin with $\mathbb{R}^N_{+}$.

In the present paper we discuss again problem
{\bf (P)}, but for $\alpha \in (-1,0)$. It turns out that for a certain range of the parameters $k$ and $\ell $, the problem has no solution despite the fact  that half-balls $B_{R} \cap \mathbb{R}^N _+ $ are stable (for precise meaning of stability see Section 4). More precisely our main result is the following
\begin{theorem}
\label{maintheorem}
Assume that $\alpha \in (-1,0)$, and that the conditions 
\begin{eqnarray}
\label{condk}
 & & k+N+\alpha -1 < \sqrt{(N-1)(N+\alpha -1)},
\\
\label{nonex}
 & & N(k+ N+\alpha -1) < (\ell+N+\alpha ) (N-1),
\\
\label{nec3intro}
 & &
\ell+1 \leq k +\frac{N+\alpha -1}{k+N+\alpha -1 }
\end{eqnarray}
are satisfied. Then the isoperimetric problem {\bf (P)} has no solution, 
nevertheless half-balls $B_R \cap \mathbb{R}^N _{+} $ are stable.
\end{theorem}
Note that the conditions (\ref{condk}), (\ref{nonex}) and (\ref{nec3intro}) are satisfied in the model case $k=\ell =0$.

The delicate part of the proof of Theorem 1.1 is to find a stability criterion for half-balls. It is well-known - see e.g. \cite{ABCMP}, Theorem 4.1 - that an equivalent task is to determine  the best constant,
 $\mu_1 ^{\alpha} 
( \mathbb{S}^{N-1}_+)
$, in a weighted Poincar\'{e}-Wirtinger inequality on the half-sphere 
$\mathbb{S}^{N-1} _+ := \mathbb{S}^{N-1} \cap \mathbb{R}^{N}_ {+}$. 

In Section 2  we prove a compact imbedding property for some weighted spaces for functions defined on the upper half-sphere.   
 To this aim we use stereographic coordinates, since, in this coordinate system, the metric is just the conformal factor times the identity. This allows us to use  an already known compact imbedding result for weighted spaces in $\mathbb{R}^{N-1} $. 

In Section 3 we first note that 
$\mu_1 ^{\alpha} ( \mathbb{S}^{N-1}_+ )$ 
represents the first nontrivial Neumann eigenvalue of 
some self-adjoint compact operator on the half-sphere. 
In view of the imbedding result this implies that 
 $\mu_1 ^{\alpha} 
( \mathbb{S}^{N-1}_+ )
$ appears as a minimum of an appropriate Rayleigh quotient. 
Then 
we write the operator in spherical coordinates and, using separation of variables and comparing the eigenvalues of two Sturm-Liouville problems, we show that the exact value of $\mu_1 ^{\alpha}
( \mathbb{S}^{N-1}_+ )
$ is $N+\alpha -1$. 
This implies the stability of  half-spheres in view of Theorem 4.1 in \cite{ABCMP}, which holds true irrespectively of the sign of $\alpha $.

In order to prove that the problem has no solution, we show in Section 4 that the ``isoperimetric ratio'' (see \eqref{rayl1})  
for  a unit ball centered at $(0, \ldots , 0, t)$ tends  to zero when $t$ goes to infinity. 
This completes the proof of Theorem 1.1.

Our paper concludes with a few remarks on stability and nonexistence for some weighted isoperimetric problems on $\mathbb{R}^N $ in Section 5.

\section{Notation and preliminary results}
Throughout this paper the following notation will be in force:
\begin{eqnarray*}
 & & 
\mathbb{R}_{+}^{N}=\left\{ x=(x_{1},...,x_{N})\in \mathbb{R}
^{N}: x_{N}>0\right\} , \ |x|:= \sqrt{\sum_{i=1} ^N x_i ^2 },\  N\geq 2 ,
\\
 & & H= \{ x= (x_1 , \ldots , x_N ):\, x_N =0 \} ,
\\ 
 & & 
B_R (x^0):=
 \left\{ 
x \in \mathbb{R}^{N}: \, |x-x^0 |  <R 
 \right\} , 
\\
 & &
B_R:=B_R (0), \quad B_R ^+ := B_R \cap \mathbb{R}^N _+ , \quad (
x^{0} \in \mathbb{R}^{N}, \ R>0),
\\
 & & 
\mathbb{S}^{N-1}=\partial B_1 , \ 
\mathbb{S}_{+}^{N-1}:=
\mathbb{S}^{N-1}\cap \mathbb{R}_{+}^{N},
\end{eqnarray*}
$$
 \mathbb{B}_{1}:= \left\{ y = (y_1 , \ldots y_{N-1})  \in \mathbb{R}^{N-1}: \, |y| := \sqrt{\dsum
\limits_{i=1}^{N-1}y_{i}^{2}}<1 \right\}.
$$
The
stereographic projection 
$$
 \mathbb{S}_{+}^{N-1}  \ni \zeta \longmapsto  y = S(\zeta )  \in \mathbb{B} _1
$$ 
 from the south pole $
P_{S}=(0,..,0,-1)$ and its inverse are given by 
\begin{equation*}
\left\{ 
\begin{array}{ccc}
\zeta _{i}=\dfrac{2y_{i}}{\left\vert y\right\vert ^{2}+1} & \text{for} & 
1\leq i\leq N-1 \\ 
&  &  \\ 
\zeta _{N}=\dfrac{1-\left\vert y\right\vert ^{2}}{\left\vert y\right\vert
^{2}+1}
 & \text{for} & i=N
\end{array}
\right. 
\end{equation*}
and
\begin{equation*}
y_{i}=\dfrac{\zeta _{i}}{1+\zeta _{N}}\text{ \ for \ }1\leq i\leq N-1,
\end{equation*}
respectively.
As well known, in this coordinate system, see e.g. \cite{Grafakos} p. 444, the  metric on 
$\mathbb{S}^{N-1}$  
is 
\begin{equation*}
g_{ij}(y)=\left( \dfrac{2}{\left\vert y \right\vert ^{2}+1}\right) ^{2}\delta
_{ij}.
\end{equation*}
Hence $d\sigma $, the volume element  on 
$\mathbb{S}^{N-1}$, is given by 
\begin{equation*}
d\sigma  
=\sqrt{\det g_{ij}(y)}\, dy=\left( \dfrac{2}{\left\vert y\right\vert ^{2}+1
}\right) ^{N-1}\, dy.
\end{equation*} 
For any function $u: \mathbb{S}^{N-1 } _+ \to \mathbb{R} $ we define $\hat{u} : \mathbb{B}_1 \to \mathbb{R} $ by 
$$
\hat{u}(y) := u(\zeta ) , \quad (y= S(\zeta ), \, \zeta \in  \mathbb{S} ^{N-1} _+).
$$
Note that, if
 $u$ is a smooth function, then
$$
|\nabla _\mathbb{S} u (\zeta ) | = \sqrt{g^{ij} \hat{u}_{y_i} (x) \hat{u}_{y_j } (y) } 
=
|\nabla \hat{u} (y) | \cdot \frac{2}{|y|^2 +1}      , \quad (\zeta \in \mathbb{S}^{N-1} _+ ).
$$
For $\alpha \in \left(-1,  +\infty  \right) $, we consider the measure 
$d\sigma _{\alpha }$,
defined  on $\mathbb{S}_{+}^{N-1}$, given by $d\sigma$
times $\zeta_{N}^{\alpha}$.
In stereographic coordinates, such a measure takes the following
form
\begin{equation*}
d\sigma _{\alpha }=\left( \dfrac{1-\left\vert y\right\vert ^{2}}{\left\vert
y\right\vert ^{2}+1}\right) ^{\alpha }
\cdot
\left( \dfrac{2}{\left\vert
y\right\vert ^{2}+1}\right) ^{N-1}\, dy.
\end{equation*}
Define the weighted Sobolev space $W^{1,2}\left( \mathbb{S}
_{+}^{N-1};\, d\sigma _{\alpha }\right) $ as the closure of $C^{\infty }(
\mathbb{S}_{+}^{N-1})$ under the norm 
\begin{equation*}
\left\Vert u\right\Vert _{W^{1,2}\left( \mathbb{S}_{+}^{N-1};\, d\sigma
_{\alpha }\right) }^{2} :=
\left\Vert  u\right\Vert _{L^{2}\left( \mathbb{S}
_{+}^{N-1};\, d\sigma _{\alpha }\right) }^{2}
+
\left \Vert \nabla _{\mathbb{S}} u \right \Vert_{L^{2}
\left( \mathbb{S}_{+}^{N-1};\, d\sigma _{\alpha }\right) }^{2}.
\end{equation*}

\begin{theorem}
\label{Comp_Embedd}
The space $W^{1,2}\left( \mathbb{S}
_{+}^{N-1};\, d\sigma _{\alpha }\right) $ is \ compactly embedded in $
L^{2}\left( \mathbb{S}_{+}^{N-1};\, d\sigma _{\alpha }\right) .$
\end{theorem}

\begin{proof}
As already noticed the stereographic projection from the south pole of $\mathbb{S}
_{+}^{N-1}$ is just $\mathbb{B}_{1}.$ Let us first write the weighted norm of a
function in stereographic coordinates.
\begin{eqnarray*}
\left\Vert \nabla _{\mathbb{S}} u\right\Vert _{L^{2}\left( \mathbb{S}_{+}^{N-1};\, d\sigma
_{\alpha }\right) }^{2} &=&\sum_{i,j}\int_{\mathbb{B}_{1}}\left[ g^{ij}\frac{
\partial \hat{u}}{\partial y_{i}}\frac{\partial \hat{u}}{\partial y_{j}}\right] \left( 
\dfrac{2}{\left\vert y\right\vert ^{2}+1}\right) ^{N-1}\left( \dfrac{
1-\left\vert y\right\vert ^{2}}{\left\vert y\right\vert ^{2}+1}\right)
^{\alpha }\, dy \\
&=&\int_{\mathbb{B}_{1}}\left[ \left( \dfrac{2}{\left\vert y\right\vert ^{2}+1}
\right) ^{-2}\left\vert \nabla \hat{u}\right\vert ^{2}\right] \left( \dfrac{2}{
\left\vert y\right\vert ^{2}+1}\right) ^{N-1}\left( \dfrac{1-\left\vert
y\right\vert ^{2}}{\left\vert y\right\vert ^{2}+1}\right) ^{\alpha }\, dy 
\\
&=&\int_{\mathbb{B}_{1}}\left\vert \nabla \hat{u}\right\vert ^{2}\dfrac{2^{N-3} \left(
\left\vert y\right\vert +1\right) ^{\alpha } }{\left( \left\vert y\right\vert
^{2}+1\right) ^{N-3+\alpha }} \cdot
\left( 1-\left\vert y\right\vert \right) ^{\alpha } \, 
dy
\end{eqnarray*}
and 
\begin{equation*}
\left\Vert u\right\Vert _{L^{2}\left( \mathbb{S}_{+}^{N-1}\, d\sigma _{\alpha
}\right) }^{2}
=
\int_{\mathbb{B}_{1}}\hat{u}^{2}
\frac{2^{N-1}\left( |y|+1 \right) ^{\alpha }
}{
\left( |y| ^{2}+1
\right) ^{N-1+\alpha }} 
\cdot 
\left( 1- \left\vert y\right\vert \right) ^{\alpha }
\, dy.
\end{equation*}
Note that there exists $C\in \left( 0,1 \right) $ such that for any $y\in \mathbb{B}_{1}$ there holds 

\begin{equation}
\label{C1}
C\leq \dfrac{2^{N-3}\left( \left\vert y\right\vert +1\right) ^{\alpha } }{\left(
\left\vert y\right\vert ^{2}+1\right) ^{N-3+\alpha }}\leq \frac{1}{C}
\end{equation}
and

\begin{equation}
\label{C2}
C\leq 2^{N-1}\left( \dfrac{1}{\left\vert y\right\vert ^{2}+1}\right)
^{N-1+\alpha }\left( \left\vert y\right\vert +1\right) ^{\alpha }\leq \frac{1
}{C}.
\end{equation}
Now consider a bounded sequence
$ \left\{ u_{n}\right\} _{n\in \mathbb{N}} $
of functions in $W^{1,2}\left( 
\mathbb{S}_{+}^{N-1};\, d\sigma _{\alpha }\right) ,$ that is, 
\begin{equation}
\left\Vert u_{n}\right\Vert _{W^{1,2}\left( \mathbb{S}_{+}^{N-1};\, d\sigma
_{\alpha }\right) } 
\leq C
\quad
\forall n\in 
\mathbb{N}.  
\label{Bdd_W}
\end{equation}
Writing 
\begin{equation*}
d(y)=\text{dist}\left( y,\partial \mathbb{B}_{1}\right) =1-|y| , 
\end{equation*}
and using (\ref{C1}) and  (\ref{C2}) one immediately realizes that
 (\ref{Bdd_W}) is equivalent to 
\begin{equation*}
\int_{\mathbb{B}_{1}}\left\vert \nabla \hat{u}_{n}\right\vert ^{2}d(y)^{\alpha
}\, dy+\int_{\mathbb{B}_{1}}\hat{u}_{n}^{2}\, d(y)^{\alpha }\, dy\leq C.
\end{equation*}
Now using Theorem 8.8 in \cite{GO1} we deduce that, up to a not relabelled
subsequence, we have that there exists a function \thinspace \thinspace $
u\in W^{1,2}\left( \mathbb{S}_{+}^{N-1};\, d\sigma _{\alpha }\right) $ such
that 
\begin{equation*}
\int_{\mathbb{B}_{1}}\left\vert \hat{u_{n}}-\hat{u}\right\vert ^{2}\, d(y)^{\alpha }\, dy\rightarrow
0
\end{equation*}
and therefore 
\begin{equation*}
u_{n}\rightarrow u\text{ strongly in }L^{2}\left( \mathbb{S}
_{+}^{N-1};\, d\sigma _{\alpha }\right) .
\end{equation*}
\end{proof}

\begin{theorem}
\label{WeighPoin}
The following Weighted Poincar\'{e}
inequality holds true
\begin{equation}
\left\Vert u-\frac{1}{\sigma _{\alpha }\left( \mathbb{S}_{+}^{N-1}\right) }
\int_{\mathbb{S}_{+}^{N-1}}u\, d\sigma _{\alpha }\right\Vert _{L^{2}\left( 
\mathbb{S}_{+}^{N-1};\, d\sigma _{\alpha }\right) }
\leq 
C \left\Vert \nabla _\mathbb{S}
u\right\Vert _{L^{2}\left( \mathbb{S}_{+}^{N-1};\, d\sigma _{\alpha }\right) },
\label{WP}
\end{equation}
where $C\in (0,+\infty )$ is a constant which does not depend on $u$.
\end{theorem}

\begin{proof}
One can obtain the proof repeating the arguments of
the classical one for the unweighted case (see, e.g., \cite{LL}, Th. 8.11, page 218). We
include it for reader's convenience.  Assume, arguing by contradiction,
that there exists a sequence $\left\{ u_{k}\right\} _{k\in \mathbb{N}
}\subset W^{1,2}\left( \mathbb{S}_{+}^{N-1};d\sigma _{\alpha }\right) $ such
that 
\begin{equation*}
\left\Vert u_{k}-\frac{1}{\sigma _{\alpha }\left( \mathbb{S}
_{+}^{N-1}\right) }\int_{\mathbb{S}_{+}^{N-1}}u_{k}\, d\sigma _{\alpha
}\right\Vert _{L^{2}\left( \mathbb{S}_{+}^{N-1};\, d\sigma _{\alpha }\right)
}
\geq 
k\left\Vert \nabla _S u_{k}\right\Vert _{L^{2}\left( \mathbb{S}
_{+}^{N-1};\, d\sigma _{\alpha }\right) }
\end{equation*}
Consider now the normalized sequence 
\begin{equation*}
v_{k}:=
\frac{u_{k}-\dfrac{1}{\sigma _{\alpha }\left( \mathbb{S}
_{+}^{N-1}\right) }\diint_{\mathbb{S}_{+}^{N-1}}u_{k}\, d\sigma _{\alpha }}
{\left\Vert u_{k}-\dfrac{1}{\sigma _{\alpha }\left( \mathbb{S}
_{+}^{N-1}\right) }\diint_{\mathbb{S}_{+}^{N-1}}u_{k}\, d\sigma _{\alpha
}\right\Vert _{L^{2}\left( \mathbb{S}_{+}^{N-1};\, d\sigma _{\alpha }\right) }}
\,\text{\ } \quad \ \forall k\in \mathbb{N}.
\end{equation*}
Clearly 
\begin{equation}
\int_{\mathbb{S}_{+}^{N-1}}v_{k}\, d\sigma _{\alpha }=0\,,\text{ \ }\left\Vert
v_{k}\right\Vert _{L^{2}\left( \mathbb{S}_{+}^{N-1};\, d\sigma _{\alpha
}\right) }=1\text{ \ and }\left\Vert \nabla _S v_{k}\right\Vert _{L^{2}\left( 
\mathbb{S}_{+}^{N-1};\, d\sigma _{\alpha }\right) }\leq \frac{1}{k}  \label{v_k}
\end{equation}
for any $k\in \mathbb{N}.$

Thanks to Theorem \ref{Comp_Embedd} we have that there
exists a function $v\in W^{1,2}\left( \mathbb{S}_{+}^{N-1};\, d\sigma _{\alpha
}\right) $ such that, up to a subsequence, 
\begin{equation*}
v_{k}\rightarrow v\text{ strongly in }L^{2}\left( \mathbb{S}
_{+}^{N-1};\, d\sigma _{\alpha }\right) .
\end{equation*}
Finally from (\ref{v_k}) we deduce that 
\begin{equation*}
\int_{\mathbb{S}_{+}^{N-1}}v\, d\sigma _{\alpha }=0, \quad 
\left\Vert
v\right\Vert _{L^{2}\left( \mathbb{S}_{+}^{N-1};\, d\sigma _{\alpha }\right)
}=1,
\quad \mbox{and }\ \nabla _{\mathbb{S}} v =0 \ \mbox{ a.e. on $\mathbb{S} ^{N-1}_+ $,}
\end{equation*}
which is impossible.
\end{proof}

\begin{remark} Note the aim of the next Section is to find the
best constant in (\ref{WP}).
\end{remark}

Using Theorem \ref{Comp_Embedd}  and Theorem \ref{WeighPoin} we immediately deduce the following

\bigskip
\begin{theorem}
\label{V_alpha}
Let
\begin{equation*}
V_{\alpha }:=\left\{ u\in W^{1,2}\left( \mathbb{S}_{+}^{N-1};\, d\sigma
_{\alpha }\right) :\, \int_{\mathbb{S}_{+}^{N-1}}u\, d\sigma _{\alpha }=0\right\} .
\end{equation*}
Every sequence 
$
\left\{ u_{n}\right\} _{n\in \mathbb{N}}\subset V_{\alpha } $ such that
$$
\left\Vert
\nabla _\mathbb{S} u_{n}\right\Vert _{L^{2}\left( \mathbb{S}_{+}^{N-1};\, d\sigma _{\alpha
}\right) }\leq C\text{ \ }\forall n\in \mathbb{N}
$$
for some $C\in (0,+\infty ),$ admits a subsequence, still denoted by $
u_{n},$ such that 
\begin{equation}
\left\Vert u_{n}-u\right\Vert _{L^{2}\left( \mathbb{S}_{+}^{N-1};\, d\sigma
_{\alpha }\right) }\rightarrow 0\text{ for some }u\in V_{\alpha }.
\label{CE_2}
\end{equation}
\end{theorem}

\section{An optimal weighted Wirtinger inequality}

The spherical coordinates on $\mathbb{S}_{+}^{N-1}$ are given by 
\begin{equation*}
\left\{ 
\begin{array}{cc}
\zeta _{N}= & \cos \theta _{1}\text{ \ \ \ \ \ \ \ \ \ \ \ \ \ \ \ \ \ \ \ \
\ \ \ \ \ \ \ \ \ \ \ \ \ } \\ 
\zeta _{N-1}= & \sin \theta _{1}\cos \theta _{2}\text{ \ \ \ \ \ \ \ \ \ \ \
\ \ \ \ \ \ \ \ \ \ \ \ \ } \\ 
\zeta _{N-2}= & \sin \theta _{1}\sin \theta _{2}\cos \theta _{3}\text{ \ \ \
\ \ \ \ \ \ \ \ \ \ \ \ } \\ 
. &  \\ 
. &  \\ 
\zeta _{2}= & \sin \theta _{1}\sin \theta _{2}\cdot ...\cdot \sin \theta
_{N-2}\cos \theta _{N-1} \\ 
\zeta _{1}= & \sin \theta _{1}\sin \theta _{2}\cdot ...\cdot \sin \theta
_{N-2}\sin \theta _{N-1}
\end{array}
\right. 
\end{equation*}
where 
\begin{equation*}
\theta _{1}\in \left( 0,\frac{\pi }{2}\right) ;\text{ \ }\theta
_{2},...,\theta _{N-2}\left( 0,\pi \right) ;\text{ \ }\theta _{N-1}\in
\left( 0,2\pi \right) .
\end{equation*}
Let $\Delta _{\mathbb{S}^{m}}$ be the classical Laplace Beltrami operator on 
$\mathbb{S}^{m}$.  We consider the following differential operator 
\begin{equation*}
\Delta _{\mathbb{S}^{N-1}}^{\alpha }u:=\frac{1}{\sin ^{N-2}\theta _{1}}\frac{
\partial }{\partial \theta _{1}}\left( \sin ^{N-2}\theta _{1}\cos ^{\alpha
}\theta _{1}\frac{\partial u}{\partial \theta _{1}}\right) +\frac{\cos
^{\alpha }\theta _{1}}{\sin ^{2}\theta _{1}}\Delta _{\mathbb{S}^{N-2}}u.
\end{equation*}
Note that
\begin{equation*}
 \Delta _{\mathbb{S}^{N-1}}^{0 }u =\Delta
_{\mathbb{S}^{N-1}}u.
\end{equation*}

Finally we will denote by  $\mu _{1}^{\alpha }(\mathbb{S}_{+}^{N-1})$ the
first non-trivial eigenvalue of the following problem 
\begin{equation}
\left\{ 
\begin{array}{cc}
\text{ \ \ }-\Delta _{\mathbb{S}^{N-1}}^{\alpha }u=\mu \cos ^{\alpha }\theta
_{1}u & \text{on }\mathbb{S}_{+}^{N-1} \\ 
&  \\ 
\diint_{\mathbb{S}_{+}^{N-1}}ud\sigma _{\alpha }=0.\text{ \ \ \ \ \ \ } & 
\end{array}
\right.   \label{EP_N}
\end{equation}

Note that, by Theorem  \ref{V_alpha}, $\mu _{1}^{\alpha }(\mathbb{S}
_{+}^{N-1})$ has the following variational characterization 
\begin{equation*}
\mu _{1}^{\alpha }(\mathbb{S}_{+}^{N-1})
=
\min \left\{ 
\frac{
\diint_{\mathbb{S}
_{+}^{N-1}}\left\vert \nabla _{\mathbb{S}}u\right\vert ^{2}\, d\sigma _{\alpha }
}
{\diint_{\mathbb{S}_{+}^{N-1}}u^{2}\, d\sigma _{\alpha }},\text{ with }u\in
W^{1,2}\left( \mathbb{S}_{+}^{N-1};\, d\sigma _{\alpha }\right) \backslash
\left\{ 0\right\} :\, \int_{\mathbb{S}_{+}^{N-1}}u\, d\sigma _{\alpha }=0\right\} .
\end{equation*}
Indeed, the differential operator appearing in (\ref{EP_N}) is self-adjoint
and compact.

\begin{theorem}
\label{mu1}
The following holds true:
\begin{equation*}
\mu _{1}^{\alpha }(\mathbb{S}_{+}^{N-1})=N+\alpha -1.
\end{equation*}
\end{theorem}

\begin{proof}
We start by using standard separation of variables. Hence let 
\begin{equation*}
\psi =g\left( \theta _{1}\right) f(\theta _{2},...,\theta _{N-1}):\mathbb{S}
_{+}^{N-1}\rightarrow \mathbb{R}
\end{equation*}
be an eigenfunction of problem (\ref{EP_N}) corresponding to an eigenvalue $
\mu $. A straightforward computation yields 
\begin{equation*}
-\frac{1}{g}\frac{1}{\sin ^{N-2}\theta _{1}\cos ^{\alpha }\theta _{1}}\frac{d
}{d\theta _{1}}\left( \sin ^{N-2}\theta _{1}\cos ^{\alpha }\theta _{1}\frac{
dg}{d\theta _{1}}\right) +\frac{1}{g}\frac{1}{\sin ^{2}\theta _{1}}\frac{
\Delta _{\mathbb{S}^{N-2}}f}{f}=\mu .
\end{equation*}
Since, see \cite{Chavel} and \cite{Muller}, 
\begin{equation*}
\frac{\Delta _{\mathbb{S}^{N-2}}f}{f}=\text{Constant \ }\Leftrightarrow \text{\ \ }
\frac{\Delta _{\mathbb{S}^{N-2}}f}{f}=k\left( k+N-3\right) ,\text{ with }k
\in \mathbb{N} \cup
\left\{ 0\right\} ,
\end{equation*}
we have 
\begin{equation}
-\frac{1}{\sin ^{N-2}\theta _{1}\cos ^{\alpha }\theta _{1}}\frac{d}{d\theta
_{1}}\left( \sin ^{N-2}\theta _{1}\cos ^{\alpha }\theta _{1}\frac{d}{d\theta
_{1}}g\right) +\frac{k\left( k+N-3\right) }{\sin ^{2}\theta _{1}}g=g\mu .
\label{SL}
\end{equation}
Let us denote with $ \left\{ \mu _{k} \right\} _{k\in \mathbb{N}_{0}} $  the sequence of eigenvalues of the
Sturm-Liouville problem (\ref{SL}).

\noindent We claim that 
\begin{equation}
\mu _{0}>(N-1)(1-\alpha ).  \label{Claim_mu_0}
\end{equation}
Clearly the first  ``radial" eigenfunction, $g_{0}(\theta _{1})$, of (\ref
{EP_N}) corresponds to $k=0$. Since $g_{0}(\theta _{1})$ has exactly two
nodal domains there exists $\widehat{\theta }\in \left( 0,\frac{\pi }{2}
\right) $ such that 
\begin{equation*}
g_{0}(\theta _{1})=0\text{ if and only if }\theta =\widehat{\theta }.
\end{equation*}
Therefore 
\begin{equation*}
\mu _{0}=\lambda _{1}(\widehat{\theta }),
\end{equation*}
where $\lambda _{1}\left( \widetilde{\theta }\right) $ is the first
eigenvalue of the following Dirichlet problem 
\begin{equation}
\left\{ 
\begin{array}{ccc}
-\Delta _{\mathbb{S}^{N-1}}^{\alpha }v=\lambda \cos ^{\alpha }\theta _{1}v & 
\text{on} & \mathbb{S}_{+}^{N-1}\cap \left\{ 0<\theta _{1}<\widetilde{\theta 
}\right\}  \\ 
&  &  \\ 
v=0 & \text{on} & \partial \left[ \mathbb{S}_{+}^{N-1}\cap \left\{ 0<\theta
_{1}<\widetilde{\theta }\right\} \right] .
\end{array}
\right.   \label{EP_D}
\end{equation}
Since, as well known, the Dirichlet eigenvalues are monotone with respect to
the inclusion of sets, we have 
\begin{equation*}
\lambda _{1}\left( \widehat{\theta }\right) >\lambda _{1}\left( \frac{\pi }{2
}\right) .
\end{equation*}
Let us conclude the proof of the claim by showing that 
\begin{equation*}
\lambda _{1}\left( \frac{\pi }{2}\right) =(N-1)(1-\alpha ).
\end{equation*}
A straightforward computation shows that 
\begin{equation*}
\psi _{0}(\theta _{1}):=\cos ^{1-\alpha }\theta _{1}
\end{equation*}
is an eigenfunction of problem (\ref{EP_D}) with $\widetilde{\theta }=\frac{
\pi }{2},$ corresponding to the eigenvalue $(N-1)(1-\alpha )$. Indeed we
have 
\begin{equation*}
\left. -\frac{1}{\sin ^{N-2}\theta _{1}\cos ^{\alpha }\theta _{1}}\frac{d}{
d\theta _{1}}\left( \sin ^{N-2}\theta _{1}\cos ^{\alpha }\theta _{1}\frac{d}{
d\theta _{1}}g\right) +\frac{k\left( k+N-3\right) }{\sin ^{2}\theta _{1}}
g\right\vert _{k=0,\text{ \ }g=\cos ^{1-\alpha }\theta _{1}}=
\end{equation*}
\begin{equation*}
-\frac{1}{\sin ^{N-2}\theta _{1}\cos ^{\alpha }\theta _{1}}\frac{d}{d\theta
_{1}}\left( \sin ^{N-2}\theta _{1}\cos ^{\alpha }\theta _{1}\frac{d\psi _{0}
}{d\theta _{1}}\right) =\frac{\left( 1-\alpha \right) }{\sin ^{N-2}\theta
_{1}\cos ^{\alpha }\theta _{1}}\frac{d}{d\theta _{1}}\left( \sin
^{N-1}\theta _{1}\right) 
\end{equation*}
\begin{equation*}
=\left( 1-\alpha \right) \left( N-1\right) \frac{\sin ^{N-2}\theta _{1}\cos
\theta _{1}}{\sin ^{N-2}\theta _{1}\cos ^{\alpha }\theta _{1}}=\left(
1-\alpha \right) \left( N-1\right) \cos ^{1-\alpha }\theta _{1}.
\end{equation*}
Since $\psi _{0}(\theta _{1})$ does not change sign on $\mathbb{S}
_{+}^{N-1}\cap \left\{ 0<\theta _{1}<\frac{\pi }{2}\right\} $, it must be an
eigenfunction corresponding to \ $\lambda _{1}\left( \frac{\pi }{2}\right) $
, and the claim follows.

\noindent Now let us turn our attention to the case $k=1$, which corresponds
to the first ``angular" eigenfunction. That is an eigenfunction $\varphi $ of
problem (\ref{EP_N}) in the form 
\begin{equation*}
\varphi =g_{1}(\theta _{1})f(\theta _{2},...,\theta _{N-1})
\end{equation*}
where 
\begin{equation*}
g_{1}(\theta _{1})>0
\quad 
\forall \theta \in \left( 0,\frac{\pi }{2}
\right) .
\end{equation*}
Note that, since any eigenvalue of the problem (\ref{SL}) is simple, the
function $g_{1}(\theta _{1})$ is unique, up to a multiplicative constant.

\noindent We claim that 
\begin{equation*}
g_{1}(\theta _{1})=\sin \theta _{1}.
\end{equation*}
Indeed we have 
\begin{equation*}
-\frac{1}{\sin ^{N-2}\theta _{1}\cos ^{\alpha }\theta _{1}}\frac{d}{d\theta
_{1}}\left( \sin ^{N-2}\theta _{1}\cos ^{\alpha }\theta _{1}\frac{dg_{1}}{
d\theta _{1}}\right) +\frac{N-2}{\sin ^{2}\theta _{1}}g_{1}
\end{equation*}
\begin{equation*}
=-\frac{1}{\sin ^{N-2}\theta _{1}\cos ^{\alpha }\theta _{1}}\frac{d}{d\theta
_{1}}\left( \sin ^{N-2}\theta _{1}\cos ^{\alpha +1}\theta _{1}\right) +\frac{
N-2}{\sin \theta _{1}}
\end{equation*}
\begin{equation*}
=-\frac{1}{\sin ^{N-2}\theta _{1}\cos ^{\alpha }\theta _{1}}\left( \left(
N-2\right) \sin ^{N-3}\theta _{1}\cos ^{\alpha +2}\theta _{1}-\left( \alpha
+1\right) \sin ^{N-1}\theta _{1}\cos ^{\alpha }\theta _{1}\right) +\frac{N-2
}{\sin \theta _{1}}
\end{equation*}
\begin{equation*}
=-\left( N-2\right) \frac{\cos ^{2}\theta _{1}}{\sin \theta _{1}}+\left(
\alpha +1\right) \sin \theta _{1}+\frac{N-2}{\sin \theta _{1}}
\end{equation*}
\begin{equation*}
=\left(
N+\alpha -1\right) \sin \theta _{1}
=\left( N+\alpha -1\right) g_{1}(\theta
_{1}).
\end{equation*}
The claim is proved. 

Gathering the above estimates, taking into account that $\alpha \in \left(
-1,0\right) $, we have 
\begin{equation*}
\mu _{0}=\lambda _{1}(\widehat{\theta })>\lambda _{1}\left( \frac{\pi }{2}
\right) =(N-1)(1-\alpha )=-N\alpha +N+\alpha -1>N+\alpha -1=\mu _{1}.
\end{equation*}
\end{proof}

\begin{remark}
By equality (4.11) of \cite{ABCMP}, we have just proven that, the second
variation of the perimeter w.r.t. volume-preserving smooth perturbations at
the half ball is nonnegative for $\alpha \in (-1 , +\infty )$. Note that in 
\cite{BCM}, see Proposition 2.1, the case of nonnegative $\alpha $ is
addressed. 
\end{remark}

\section{An isoperimetric problem in the half space and a curious example 
}
In this section we consider an isoperimetric problem that we have studied in \cite{ABCMP}, but we will change the range of one of the parameters in it. 
\\
\noindent Let $k$, $\ell$ and $\alpha $ be real
numbers satisfying 
\begin{eqnarray}
\label{alphaass}
 & & \alpha >-1,
\\
\label{lass}
 & & \ell+N+ \alpha >0,
\\
\label{kass}
 & & k+N+\alpha >0.
\end{eqnarray}
We define a measure $\mu _{\ell, \alpha}$ on $\mathbb{R}^{N}_+ $  by 
\begin{equation}
d\mu _{\ell, \alpha}(x)=|x|^{\ell} x_N^\alpha\,dx.  
\label{dmu}
\end{equation}
If $M \subset $ ${\mathbb R}^{N}_{+}$ is a  measurable set with finite 
$\mu _{\ell, \alpha} $-measure, then we define $M^{\star }$, the 
\\
$\mu_{\ell,\alpha }$-symmetrization of $M$,
 as 
\begin{equation}
\label{musymm}
M^{\star } := B_R ^+ ,
\end{equation}
where $R$ is given by 
\begin{equation}
\mu_{\ell, \alpha} 
\left( B_{R} ^+  \right) =
\mu _{\ell, \alpha} \left( M\right) = \int_M d\mu _{\ell, \alpha} (x) .  
\label{mu_(M)}
\end{equation}
Following \cite{MP}, the {\sl $\mu _{k, \alpha}$--perimeter relative to $\mathbb{R}^N _+$ \/} of a measurable set $M $ of locally finite perimeter - henceforth simply called the {\sl relative  $\mu _{k, \alpha}$--perimeter} -  is given by 
\begin{equation}
\label{perdef}
P_{\mu _{k, \alpha}}(M , \mathbb{R}^N _+ ):= \int_{\partial M\cap \mathbb{R} ^N_+ } \, x_N^\alpha |x|^{k}
 \, \mathcal{H}_{N-1} (dx) .
\end{equation}
Here and throughout, $\partial M$ and ${\mathcal H} _{N-1} $ will denote the essential boundary of $M$ and  $(N-1)$-dimensional Hausdorff-measure, respectively. 

We will call a set $\Omega \subset \mathbb{R}^N _+ $ a $C^n$-set,  ($n\in \mathbb{N} $), 
if for every $x^0 \in \partial \Omega \cap \mathbb{R}^N _+ $, 
there is a number $r >0$ such that 
$B_r (x^0 ) \cap \Omega $
has exactly one connected component and $B_r (x^0 ) \cap \partial \Omega $ 
is the graph of a $C^n$-function on an open set in $\mathbb{R} ^{N-1} $.  

We consider a one-parameter family $\{ \varphi _t \}_{t} $ of $C^n $-variations
$$
\mathbb{R}^N _+  \times (-\varepsilon , \varepsilon ) \ni (x,t) \longmapsto \varphi (x,t) \equiv \varphi _t (x) \in \mathbb{R}^N _+ ,
$$ 
with $\varphi (x, 0)=x$, for any $x\in \mathbb{R}^N_+ $. 
The measure and perimeter functions
of the variation are $m(t) := \mu _{\ell ,\alpha }(\varphi _t (\Omega ))$  and $p (t) := P_{\mu _{k,\alpha }} (\varphi _t (\Omega ))$, respectively. We say that the variation $\{ \varphi _t  \} _t$ of $\Omega $ is {\sl measure-preserving} if $m(t)$ is constant for any small $t$. We
say that a $C^1$-set $\Omega $ is {\sl stationary} if $p'(0) = 0$ for any measure-preserving $C^1 $-variation.
Finally, we call a $C^2$-set $\Omega$ {\sl stable} if it is stationary and $p''(0) \geq  0$ for any
measure-preserving $C^2 $-variation of $\Omega $.

If $M $ is any measurable subset of $\mathbb R^{N}_{+}$, with $0<\mu _{\ell,\alpha} (M)<+\infty $, we set
\begin{equation}
\label{rayl1}
{\mathcal R}_{k,\ell,N, \alpha} (M) := 
\frac {P_{ \mu_{k , \alpha}}  (M) }
{  \left( \mu_{\ell,\alpha} (M) \right)^{(k+N+\alpha-1)/(\ell+N+\alpha)} }. 
\end{equation}
Finally, we define 
\begin{equation}
\label{isopco}
C_{k,\ell,N, \alpha}^{rad} := {\mathcal R}_{k,\ell,N, \alpha}
(B_1 ^+ ).
\end{equation}  
We study the following isoperimetric problem: 
\\[0.1cm]
{\sl Find the constant $C_{k,\ell,N, \alpha} \in [0, + \infty )$, such that}
\begin{eqnarray}
\label{isopproblem}
 & & C _{k,\ell,N, \alpha}  :=  
\inf \{ {\mathcal R}_{k,\ell,N, \alpha} (M):\, 
\mbox{{\sl $M$ is a measurable set with locally finite perimeter}}  
\\
\nonumber
 & & \qquad \quad \quad \ \mbox{ {\sl and $0<\mu _{\ell,\alpha} (M) <+\infty $}} \} . 
\end{eqnarray}   
Moreover, we are interested in conditions on $k$, $\ell$ and $\alpha$ such that
\begin{equation}
\label{isoradial}
{\mathcal R}_{k,\ell,N, \alpha} (M) \geq {\mathcal R}_{k,\ell,N, \alpha} (M^{ \star} ) 
\end{equation}
holds for all measurable sets $M\subset {\mathbb{R} ^N_+ }$ with  $ 0<\mu _{\ell,\alpha}(M)<+\infty $ and locally finite perimeter. 
\\[0.1cm]
Let us begin with some immediate observations.
\\ 
The conditions (\ref{alphaass}), (\ref{kass}) and (\ref{lass}) have been made to ensure that the integrals (\ref{mu_(M)}) and (\ref{perdef}) converge. The cases $\alpha =0$ and $\alpha >0$ were analysed in the articles \cite{ABCMP2017} and \cite{ABCMP}, respectively. Here we are only interested in the case 
$$
\alpha \in (-1, 0), 
$$
that is, 
our weight functions are {\sl singular} on the hyperplane $\{ x_N =0\} $. Hence our definition (\ref{perdef}) gives a {\sl relative } perimeter: boundary parts contained in the hyperplane $H $ do not count. 
\\ 
The functional ${\mathcal R}_{k,\ell,N, \alpha } $  
has the following homogeneity properties, 
\begin{eqnarray}
\label{hom1}
 {\mathcal R}_{k,\ell,N, \alpha } (M ) & = & {\mathcal R}_{k,\ell,N, \alpha } (tM ) ,
\end{eqnarray}
where  $t>0$, $M $ is a measurable set with $0<\mu_{\ell, \alpha} (M)<+\infty $ and
$tM := \{t \zeta:\, \zeta\in M \} $, and there holds   
\begin{equation}
\label{isopconst2}
C_{k,\ell,N, \alpha} ^{rad} = {\mathcal R}_{k,\ell,N, \alpha} (B_1 ^+  ).
\end{equation}  
Hence we have that 
\begin{equation}
\label{relCC}
C_{k,\ell,N, \alpha} \leq C_{k,\ell,N, \alpha} ^{rad} ,
\end{equation}  
and (\ref{isoradial}) holds if and only if 
\begin{equation}
\label{CCrad}
C_{k,\ell,N,\alpha } = C_{k,\ell,N,\alpha } ^{rad} .
\end{equation}
We have the following
\begin{lemma}
\label{3.1}
Let $ \alpha \in (-1,0)$. 
Then a necessary condition
for the existence of minimizers of problem {\bf (P)} is
\begin{equation}
\label{nec1}
kN \geq \ell(N-1) -\alpha .
\end{equation}
\end{lemma}
\begin{proof}
In the following we write for any two continuous functions $f,g:  
(0, +\infty ) \to (0,+\infty )$,
$$
f \simeq g \quad \Longleftrightarrow \quad c_1 f(t) \leq g(t) \leq c_2 g(t) \ \ \forall t\in [1, +\infty ),
$$
for some constants $0<c_1 <c_2 $.
\\
Assume that (\ref{nec1}) does not hold. 
Let $\Omega (t) := B_1 (0, \ldots , 0 , t) $, ($t>1$). 
Then we have 
$$
\mathcal{R}_{k,\ell,N,\alpha } (\Omega (t) ) \simeq t^{\alpha +k - (k+N+\alpha -1)(\alpha +\ell)/(\ell+N+\alpha )}  
.$$
Since $
KN < \ell(N-1) -\alpha $, it follows that
$$
\lim_{t\to \infty } \mathcal{R}_{k,\ell,N,\alpha } (\Omega (t) ) =0,
$$
that is, problem {\bf (P)} has no minimizer.
\end{proof}

\begin{remark}
{\bf (a)} Observe that 
(\ref{nec1}) is equivalent to
\begin{equation}
\label{nec1-1}
N(k+ N+\alpha -1 ) \geq (N-1) (\ell+N+\alpha ).
\end{equation}
Note also that (\ref{nec1}) is not satisfied if
$$
k=\ell=0,
$$
that is, problem {\bf (P)} has no minimizer in this case.
\\
{\bf (b)} Using trial domains   
$$
\Omega (t) = B_1 (t, 0,\ldots ,0), 
$$
and proceeding 
similarly as in the above proof, leads to another necessary condition for existence of minimizers of {\bf (P)}, namely:
\begin{equation}
\label{nec2}
k(N+ \alpha )\geq \ell(N+\alpha -1).
\end{equation}
This necessary condition has been obtained in the case $\alpha \geq 0$ in \cite{ABCMP}, Lemma 4.1. 
Note that in our case, $\alpha \in (-1,0)$, it holds true, too. However, if $\alpha \in (-1,0)$, then (\ref{nec1}) is more restrictive than (\ref{nec2}).  
\end{remark}

\begin{lemma}
\label{Lemma_nec3}

A necessary condition for radiality of the minimizers of problem {\bf (P)} is
\begin{equation}
\label{nec3}
\ell+1 \leq k +\frac{N+\alpha -1}{k+N+\alpha -1 } .
\end{equation}
Moreover, if (\ref{nec3}) is satisfied, then half-balls $B_R^+  $, ($R>0$), are stable for problem {\bf (P)}.  

\end{lemma}

\begin{proof}
This property  has been obtained for the case $\alpha \geq 0$ in \cite{ABCMP}, Theorem 4.1. The proof essentially depends on the fact that the first eigenvalue of the problem (\ref{EP_N}),  $\mu _1 ^{\alpha }( \mathbb{S} ^{N-1} _+ )$ is equal to $N+\alpha -1$. As we have proven above in Theorem \ref{mu1}, that property still holds for $\alpha \in (-1,0)$. Hence the proof of \cite{ABCMP} carries over to our case.
\end{proof}

Now we are the position to prove our main result.

\vspace{0.2cm}
{\bf Proof of Theorem \ref{maintheorem}: } Non-existence follows from Lemma \ref{3.1}, while the fact that half-balls are stable  for problem {\bf (P)} follows from Lemma \ref{Lemma_nec3} - see also \cite{ABCMP}, Theorem 4.1 and Theorem 5.2 for the special case $N=2$, $k=\ell =0$.
$\hfill \Box $

\begin{remark}
Observe that for each  $\alpha \in (-1,0)$, the set of pairs $(k,\ell)$ satisfying the conditions (\ref{nonex}) and (\ref{nec3}) is non-empty in view of (\ref{condk}). In particular, it contains the point $(0,0)$. 
\end{remark}

We conclude with a result that has been obtained for the cases $\alpha =0$ and $\alpha >0$ in the papers \cite{ABCMP2017} and \cite{ABCMP}, respectively. 

\begin{theorem}
Let $k\geq \ell + 1 $ and $\alpha \in (-1,0)$. Then (\ref{CCrad}) holds. Moreover, if $k> \ell + 1 $ and
\begin{equation}
\label{R=Crad}
{\mathcal R} _{k,\ell,N,\alpha } (M) = C^{rad} _{k,\ell , N,\alpha } 
\quad \mbox{for some measurable set $M \subset \mathbb{R}^N_+ $
with $0 < \mu _{\ell, \alpha } (M) < +\infty $,}
\end{equation} 
then $M = B_R ^+ $
for some $R > 0$.
\end{theorem}

For the proof we need a property that has been known for the cases $\alpha \geq 0$, see \cite{ABCMP}, Lemma 4.1. The proof carries over to our situation without changes. 

\begin{lemma}
Let $k,\ell $ and $\alpha $ be as above and $\ell ' \in (-N-\alpha , \ell )$.  
Further, 
assume that 
$C_{k,\ell, N,\alpha } = C_{k,\ell ,N, \alpha } ^{rad} $. Then we also
have $C_{k,\ell ', N,\alpha } = C_{k,\ell ',N, \alpha } ^{rad} $.
Moreover, if $ {\mathcal R}_{k,\ell ' ,N,\alpha }(M) = C_{k,\ell ' , N, \alpha } ^{rad} $
for some measurable set $M \subset \mathbb{R}^N_+$, with
$0< \mu _{\ell ', \alpha }(M) < +\infty $, then $M = B _R ^+ $
for some $R>0$.
\end{lemma}

{\sl Proof of Theorem 4.1:} We proceed similarly as in \cite{ABCMP}, proof of Theorem 4.1. The idea is to use Gauss' Divergence Theorem.  
We split into two cases.
\\
{\bf 1.} Assume that $k= \ell  + 1$, and let $\Omega $  a $C^1 $-set.  Define the domain
$$
\widetilde{\Omega} := \Omega \cup (H\cap \partial \Omega )
 \cup \{ x= (x_1 , \ldots , -x_N ) :\, x\in \Omega \} .
$$
Then we have in view of the assumptions (\ref{alphaass}), (\ref{kass}) and (\ref{lass}),
\begin{eqnarray}
\label{per1}
2 \int_{\partial \Omega \cap \mathbb{R}^N _+ } |x| ^{\ell} x _N ^{\alpha } (x \cdot \nu ) \, \mathcal{H}_{N-1} (dx ) & = &
 \int_{\partial \widetilde{\Omega }  } |x| ^{\ell} x _N ^{\alpha } (x \cdot \nu ) \, \mathcal{H}_{N-1} (dx ) , 
\\
\nonumber
 & & (\nu : \mbox{ exterior unit normal } ),
\\ 
\label{volume1}
2 \int_{\Omega } |x| ^{\ell} x _N ^{\alpha }  \, dx & = & 
 \int_{\widetilde{\Omega} } |x| ^{\ell} x _N ^{\alpha }  \, dx .
\end{eqnarray}
Furthermore, Gauss' Divergence Theorem yields
\begin{eqnarray}
\label{gauss}
\int_{\widetilde{\Omega } } 
|x|^{\ell} x_N ^{\alpha } \, dx
 & = &
 \frac{1}{\ell +N +\alpha } 
\int_{\widetilde{\Omega } } 
\mbox{div}\, \left( |x|^{\ell} x_N ^{\alpha } x \right) \, dx 
\\
\nonumber 
 & = &
\int_{\partial \widetilde{\Omega } } 
|x|^{\ell} x_N ^{\alpha } (x\cdot \nu ) {\mathcal H} _{N-1} (dx) 
\\
\nonumber
 & \leq & 
\int_{\partial \widetilde{\Omega } } 
|x|^{\ell +1 } x_N ^{\alpha } {\mathcal H} _{N-1} (dx),
\end{eqnarray}
with equality for $\widetilde{\Omega} = B_R $. Using this, (\ref{per1}), and (\ref{volume1}), we obtain (\ref{CCrad})  for $C^1 $-sets when $k= \ell +1$, and then by approximation also for sets with locally finite perimeter.
\\
{\bf 2.} Let $k> \ell +1$. Then, using Lemma 4.3 and the result for $k=\ell +1$, we again obtain (\ref{CCrad}),  and (\ref{R=Crad}) can hold only
if $M = B_R ^+ $.
$\hfill \Box $

\section{Some remarks on isoperimetric problems 
on $\mathbb{R}^N $}

Ideas as they were used in the last section are useful in other situations as well. In this section we are interested in criteria for nonexistence and nonradiality of solutions to some weighted isoperimetric problems on $\mathbb{R}^N$.
More results to  these and related questions can be found in the papers \cite{MP}, \cite{DHHT}, \cite{Howe},
\cite{morgan2} and in \cite{morganblog}. 
\\ 
Let $f,g$ be two positive functions on $ \mathbb{R}^N  $ with $g$ locally integrable and $f$ lower semi-continuous. For any measurable set $M\subset \mathbb{R}^N$ we define its weighted  measure and perimeter by 
\begin{eqnarray}
\label{mumeas2}
|M|_g  & := & \int_M g(x)\, dx , \quad \mbox{and}  
\\
\label{perdef2}
P_f (M) & := & \int_{\partial M}   f(x) \, \mathcal{H}_{N-1}  (dx)  .
\end{eqnarray}
Then $C^n$-sets, stationary and stable sets are defined analogously as in Section 4, replacing $\mathbb{R}_+ ^N $, $P_{\mu_{k,\alpha} } (M)$ and $\mu_{\ell, \alpha } (M)$ by $\mathbb{R}^N $, $P_f (M)$ and $|M|_g $, respectively.   
\\
We consider the isoperimetric problem 
\begin{equation}
\label{isop2}
\mbox{{\sl 
Find }} 
\ \ \ \inf \Big\{ P_f (M): \,  \, M \ 
\mbox{{\sl  has locally finite perimeter and }} 
|M|_g =d \ \Big\} , \ \ (d>0).
\end{equation}
Let us first assume that $f$ and $g$ are equal and radial, that is, there is a function $h: [0, +\infty) \to (0,\infty )$ such that
\begin{equation}
\label{f=g}
 f(x) = g(x) = h(|x|) \quad \forall x\in \mathbb{R}^N .
\end{equation}  
It has been known for some time - see for instance \cite{bayleetal}, Corollary 3.11 - that if $h\in 
C^2 (0,+\infty )$, and if $\log h$ is convex (equivalently, if $h$ is log-convex) then balls centered at the origin are stable for the isoperimetric problem (\ref{isop2}). Recently  G. Chambers, see \cite{Cham} proved the beautiful {\sl Log-convex Theorem}:
\\
{\sl  If $f=g$, $f\in C^1$ and  $h$ is log-convex, then balls centered at the origin solve problem (\ref{isop2}).}
\\
Note that the smoothness assumption for $f$ at zero in the theorem forces $h$ to be non-decreasing. 
\\
We will show below that the situation is different when $h$ is log-convex, but decreasing on some interval.  

\begin{lemma}
Assume that $f$ satisfies (\ref{f=g}), where $r\mapsto h(r)$ is log-convex and strictly decreasing for $r\in (0,R_0 )$, for  some $R_0 >0 $. Then there exists a number $d_0 >0$, which depends only on $R_0 $, such that for any $d\in (0,d_0 ]$,  balls centered at the origin with measure $d$ are not isoperimetric for problem (\ref{isop2}).
\end{lemma}

\begin{proof} For any $d>0 $ choose positive numbers $R(d)$, $\rho (d) $, such that
\begin{eqnarray}
\label{R(d)}
|B_{R(d)} |_f  & = &  |B_{\rho (d)} (y(d))|_f  =d,
\\
\nonumber
 & & \mbox{
where $
y(d) = (R_0 -\rho (d), 0, \ldots , 0).
$
}
\end{eqnarray}  
If $d$ is small enough - say $d\in (0, d_0 ]$ - then we have that
\begin{eqnarray}
\label{smalld}
 & & R(d) \leq R_0 -2 \rho (d) \quad \mbox{and}
\\
\label{cruc1}
 & & 
\left[ h(R_0 - 2d) \right]  ^N <
h(R(d)) \left[ h(R_0 ) \right] ^{N-1} .
\end{eqnarray}
From (\ref{R(d)}) we find, using the monotonicity of $h$,
$$
\omega _N h(R(d)) (R(d))^N > \omega _N h(R_0 ) (\rho (d))^N,
$$
that is,
\begin{equation}
\label{cruc2}
\rho (d) < \left( \frac{h(R(d))}{h(R_0 )} \right) ^{1/N} R(d).
\end{equation}  
Hence
the monotonicity of $h$,  (\ref{smalld}) (\ref{cruc1}) and (\ref{cruc2}) yield  
\begin{eqnarray}
\label{smallerper}
P_f (B_{\rho (d)} (y(d)) & < & N\omega _N h(R_0 - 2\rho (d)) (\rho (d) )^{N-1}  
\\
\nonumber
 & < & N\omega _N h(R_0 - 2\rho (d)) \cdot \left( \frac{h(R(d))}{h(R_0 )} \right) ^{(N-1)/N} \cdot (R(d)) ^{N-1}
\\
\nonumber
 & < & N\omega _N h(R(d)) (R(d))^{N-1} 
\\
\nonumber
 & < &  
P_f (B_{R(d)} ). 
\end{eqnarray}
This proves the Lemma.
\end{proof} 

We conclude this section with a non-existence result. 

\begin{theorem}
Assume that 
$f$ and $g$ satisfy 
\begin{eqnarray}
\label{homogeneous1}
& &  f(x) \leq c_1 |x|^{- \alpha } \ \  \mbox{and}
\\
\label{homogeneous2}
 & & g(x) \geq c_2 |x|^{-\beta } \quad \mbox{for $|x| \geq R_1 $,}
\end{eqnarray}
where $\alpha $, $\beta $, $R_1$, $c_1$ and $c_2 $ are positive numbers
and
\begin{equation}
\label{alphabeta}
\beta \leq N \ \ \ \mbox{and }\ \ \alpha > \frac{N-1}{N} \cdot \beta .
\end{equation}  
Then the isoperimetric problem (\ref{isop2}) has no solution.
\end{theorem}

\begin{proof}
Fix $d>0$, and set $z(t):= (t, 0, \ldots ,0)$ for every $t>0$. Choose $R(t)>0$ such that 
\begin{equation}
\label{measB}
|B_{R(t)} (z(t))|_g =d.
\end{equation}
In view of (\ref{homogeneous2}) this implies that
\begin{equation}
\label{limt-R}
\lim_{t\to +\infty } (t-R(t)) = +\infty .
\end{equation}
When $t$ is large enough - say $t\geq t_0 $ -  
assumption (\ref{homogeneous2}) and (\ref{limt-R}) yield
\begin{eqnarray}
\label{est1}
|B_{R(t)} (z(t)) |_g 
 & = & \int_{B_{R(t)} (z(t)) } g(x)\, dx = t^{N } \int_{B_{R(t)/t} (z(1))} g(ty)\, dy 
\\
\nonumber 
 & \geq  & c_2 t^{N-\beta } \int_{B_{R(t)/t} (z(1))} |y|^{-\beta } \, dy .
\end{eqnarray}
Now from (\ref{est1}) we obtain the following alternative:
\begin{eqnarray}
\label{limR}
 & & \mbox{If $\beta <N$, then $ 
\lim_{t\to +\infty } \frac{R(t)}{t} =0$, and}
\\
\label{R/t}
 & & \mbox{if $\beta =N$, then $\frac{R(t)}{t} \leq 1-\delta $ for $t\geq t_0 $, for some $\delta \in (0,1)$.}
\end{eqnarray} 
Further, from (\ref{measB}) we have
\begin{equation}
\label{estimate1}
d\geq \omega _N (R(t))^N c_2 (t+ R(t)) ^{-\beta } . 
\end{equation} 
Using this, (\ref{limR}), (\ref{R/t}), (\ref{alphabeta}) and again (\ref{homogeneous1}), leads to 
\begin{eqnarray}
\label{perimeterbelow} 
P_f (B_{R(t)} (z(t))) 
 & = & \int_{\partial B_{R(t)} (z(t))} f(x) \, \mathcal{H} _{N-1} (dx) 
\\
\nonumber
 & \leq & c_1 (t-R(t)) ^{-\alpha } N\omega _N (R(t))^{N-1}
\\
\nonumber
 & \leq & c_1 (t-R(t)) ^{-\alpha } N\omega _N 
\left( 
\frac{
d(t+R(t))^{\beta }
}{
c_2 \omega _N 
} 
\right) ^{(N-1)/N} 
\\
\nonumber 
 & & \longrightarrow 0 \quad \mbox{as $t\to +\infty$.}
\end{eqnarray}
The Theorem is proved. 
\end{proof}

\begin{remark}
The case that $f(x) = |x|^{-\alpha} $, $g(x)= |x|^{-\beta }$, ($x\in \mathbb{R} ^N $), with $\beta <N$, was treated in \cite{ABCMP2017}, Lemma 4.1. See also \cite{DHHT}, Proposition 7.3 for the special case $f(x)=g(x)=|x|^{-\beta } $.  
\end{remark} 
\noindent
{\bf Acknowledgement:} This work was supported by Leverhulme Trust (UK), ref. P1415-15. The authors would like to thank the 
Universities of Napoli and Swansea and the South China University of Technology (SCUT, Guangzhou) for visiting appointments and their kind hospitality.

\end{document}